\documentclass[default,12pt]{sn-jnl}
\usepackage{graphicx}%
\usepackage{multirow}%
\usepackage{amsmath,amssymb,amsfonts}%
\usepackage{amsthm}%
\usepackage{mathrsfs}%
\usepackage[title]{appendix}%
\usepackage{xcolor}%
\usepackage{textcomp}%
\usepackage{manyfoot}%
\usepackage{booktabs}%
\usepackage{algorithm}%
\usepackage{algorithmicx}%
\usepackage{algpseudocode}%
\usepackage{listings}%
\usepackage{enumitem}
\usepackage{comment}
\theoremstyle{plain}%
\newtheorem{theorem}{Theorem}
\newtheorem{proposition}[theorem]{Proposition}%
\newtheorem{lemma}[theorem]{Lemma}%

\theoremstyle{definition}%
\newtheorem{definition}[theorem]{Definition}%
\newtheorem{example}[theorem]{Example}%
\newtheorem{remark}[theorem]{Remark}%

\numberwithin{theorem}{section}
\numberwithin{equation}{section}
\begin{document}

\title[A Variation Problem for Mappings between Statistical Manifolds]{A Variation Problem for Mappings between Statistical Manifolds}

%%=============================================================%%
%% Prefix	-> \pfx{Dr}
%% GivenName	-> \fnm{Joergen W.}
%% Particle	-> \spfx{van der} -> surname prefix
%% FamilyName	-> \sur{Ploeg}
%% Suffix	-> \sfx{IV}
%% NatureName	-> \tanm{Poet Laureate} -> Title after name
%% Degrees	-> \dgr{MSc, PhD}
%% \author*[1,2]{\pfx{Dr} \fnm{Joergen W.} \spfx{van der} \sur{Ploeg} \sfx{IV} \tanm{Poet Laureate} 
%%                 \dgr{MSc, PhD}}\email{iauthor@gmail.com}
%%=============================================================%%

\author[1]{\fnm{Hitoshi} \sur{Furuhata}}
%\equalcont{These authors contributed equally to this work.}

\author*[2]{\fnm{Ryu} \sur{Ueno}}

%\author[1,2]{\fnm{Third} \sur{Author}}\email{iiiauthor@gmail.com}
%\equalcont{These authors contributed equally to this work.}

\affil*[1,2]{\orgdiv{Department of Mathematics}, 
\orgname{Hokkaido University}, 
\orgaddress{%\street{Street}, 
\city{Sapporo}, \postcode{060-0810}, 
%\state{State}, 
\country{Japan}}}

%\affil[2]{\orgdiv{Department}, \orgname{Organization}, %\orgaddress{\street{Street}, \city{City}, \postcode{10587}, %\state{State}, \country{Country}}}

%%==================================%%
%% sample for unstructured abstract %%
%%==================================%%

\abstract{
We present statistical biharmonic maps, a new class of mappings between statistical manifolds naturally derived from a variation problem. 
We give the Euler-Lagrange equation of this problem 
and prove that improper affine hyperspheres induce examples of such maps. 
}

\keywords{statistical manifolds, harmonic maps, biharmonic maps, 
affine minimal hypersurfaces, improper affine hyperspheres}

%%\pacs[JEL Classification]{D8, H51}

\pacs[Mathematical Subject Classification]{
53B12, 53A15, 58E20, 53C43}

\maketitle

%\footnotetext{\date{\today}}

\section{Introduction}\label{sec1}

A statistical manifold is nothing but a Riemannian manifold 
equipped with a torsion-free affine connection 
satisfying the Codazzi equation. 
Although the origin of this curious terminology is in information geometry, 
it is a fact that 
various research fields have provided 
interesting examples of statistical manifolds. 
At the beginning of the 20th century, 
W.~Blaschke developed the hypersurface theory in the equiaffine space. 
A pair of so-called affine fundamental form 
and the induced affine connection is a statistical structure on the hypersurface. 
Many types of research on affine hyperspheres, affine minimal hypersurfaces, 
and so on have a relation to the study of statistical structures 
(See Example \ref{Equiaffine Geometry}). 

In Riemannian geometry, harmonic maps might be understood 
as the most important research subject derived from the variation principle. 
Finding the corresponding subjects in the geometry of statistical manifolds 
is an interesting task for us. 
In fact, the Euler-Lagrange equation, the harmonic map equation, can be generalized 
in the statistical manifold setting (See \cite{MR4010286, MR3329737} for example). 
However, it is hard to find attempts to generalize them as the variation problem itself. 
In this paper, we will take notice of a certain class of mappings 
between statistical manifolds determined from a variation problem. 

Let $(M,g,\nabla^M)$ be a statistical manifold. 
By definition, the affine connection $\nabla^M$ is not necessarily 
the Levi-Civita connection of $g$. 
If so, we call the statistical manifold {\it Riemannian}. 
Assume temporarily that $M$ is compact for the sake of simplicity. 
Let $(N,h,\nabla^N)$ be a statistical manifold. 
For a smooth map $u:M\to N$, we define 
\begin{eqnarray}\label{tension}
        \tau(u)&=&        
        \tau^{(g, \nabla^M,\nabla^N)}(u)  \\
        &=&
        \mathrm{tr}_g\{(X,Y)\mapsto\nabla^N_X u_*Y-u_*\nabla^M_X Y\} 
        \in \Gamma(u^{-1}TN), \nonumber
\end{eqnarray}
and 
\begin{eqnarray}\label{eq:bienergy}
    E_2(u)
    &=&E_2^{(g, \nabla^M, h, \nabla^N)}(u)  \\
    &=&\frac{1}{2} \int_{M} \| \tau(u) \|_{h}^2 d\mu_g,  \nonumber
\end{eqnarray}
where $d\mu_g$ is the standard measure derived from the Riemanian metric $g$. 
We remark that the functional $E_2$ is defined 
by using the statistical structures of both the source and the target. 
It seems that our $E_2$ is the simplest one among such functionals.  
This functional has been known as the {\sl bi-energy}\/ for a smooth map 
in the case where both statistical manifolds are Riemannian. 
In this case, the first variation formula is given by Jiang Guoying, 
and its critical point is called a {\sl biharmonic map}\/ 
(See \cite{GuoyingJiang, MR3889042} for example). 

In the statistical manifold case, 
we obtain the first variation formula (Theorem \ref{FVF}),  
and $\tau_2(u)=0$ as the Euler-Lagrange equation, where  
\begin{eqnarray} \label{eq:bitension}
\tau_2(u)
&=& \tau_2^{(g,\nabla^M, h,\nabla^N)}(u)   \\
&=&        
\bar{\Delta}^u\tau(u)+\operatorname{div}^g
(\mathrm{tr}_g K^M)\tau(u) \nonumber\\
&&\quad 
-\sum_{i=1}^mL^N(u_*e_i,\tau(u)) u_*e_i -K^N(\tau(u),\tau(u) )
\quad 
\in \Gamma(u^{-1}TN).      \nonumber
\end{eqnarray}
The symbols will be prepared in \S 2.  
In the case where both statistical manifolds are Riemannian, 
the above equation of course restores Jiang's equation. 
We will call $\tau_2(u)$  
the {\sl statistical bi-tension field}\/ of $u$, 
and say that $u$ is a {\sl statistical biharmonic}\/ map if $\tau_2(u)=0$. 
This class might be the first one derived from  a variation problem 
for mappings between statistical manifolds. 

We illustrate simple examples of statistical biharmonic maps in \S 4. 
We remark that a map $u$ with $\tau(u)=0$ is statistical biharmonic by definition. 
We will introduce one of the interesting examples beforehand
(Theorem \ref{affineimmersion}). 
Let $f: M \to (\mathbb{R}^{m+1}, D, \det)$ be an improper affine hypersphere,  
and $(g, \nabla^M)$ the induced statistical structure on $M$ by $f$. 
Although the setting is equiaffine geometry, 
the standard connection $D$ is considered as the Levi-Civita connection 
of the Euclidean metric $g_0$, 
and the volume form $\det$ is also naturally determined by $g_0$. 
An improper affine hypersphere $f$ is statistical biharmonic 
as a map from the statistical manifold $(M, g, \nabla^M)$ 
to the Riemannian statistical manifold $(\mathbb{R}^{m+1}, g_0, D)$. 
In this case, $\tau(f)$ does not vanish, though $\tau_2(f)$ vanishes.   
Improper affine hyperspheres are important research objects in affine geometry 
(See \cite{MR1823925, NomizuSasaki} for example).  

This paper is organized as follows. 
In \S 2, we prepare the basic notion of statistical manifolds,  
particularly, Laplacian, which is useful for our first variation formula. 
\S 3 devotes the proof of the first variation formula. 
In \S 5, we give some conditions which imply $\tau=0$ from $\tau_2=0$. 
In the Riemannian setting, harmonic maps are naturally biharmonic maps, 
and in many settings, the converse is also true, 
as the Chen conjecture \cite{MR1143504} has suggested.  
We give the corresponding properties in the statistical setting. 

\vspace{0.5\baselineskip}

\section{Preliminaries}\label{sec2} %%%%%%%%%%

\subsection{Statistical manifolds}\label{subsec2.1}

Throughout this paper, all the objects are assumed to be smooth.
$M$ denotes a manifold of dimension $m \ge 2$, $C^{\infty}(M)$ the set of functions on $M$, and $\Gamma(E)$ the set of sections of a vector bundle $E$ over $M$.

In this section, we quickly fix the notation of the geometry of statistical manifolds.
Let $(g, \nabla)$ be a statistical structure on $M$,
that is, $g$ is a Riemannian metric and
$\nabla$ is a torsion-free affine connection satisfying
$(\nabla_X g)(Y,Z)=(\nabla_Yg)(X,Z)$ for $X, Y, Z \in \Gamma(TM)$.
We denote by $\nabla^g$ the Levi-Civita connection of $g$.

\vspace{0.5\baselineskip}

\begin{definition}
Let $(M,g,\nabla)$ be a statistical manifold.

(1) \ 
We set $K=K^{(g,\nabla)} \in \Gamma(TM^{(1,2)})$ by
$$
K(X,Y)= K_XY= \nabla_X Y-\nabla^g_X Y,
\quad X, Y \in \Gamma(TM).   
$$
We also denote it by $K^M$ if necessary. 

(2) \ 
We define an affine connection $\overline{\nabla}$ by
$$
Xg(Y,Z)=g(\nabla_XY,Z)+g(Y, \overline{\nabla}_XZ),
\quad X, Y, Z \in \Gamma(TM),    
$$
and call it the {\sl conjugate connection}\/ of $\nabla$ with respect to $g$. 

(3) \ 
We define the curvature tensor field $R^\nabla \in \Gamma(TM^{(1,3)})$ 
of $\nabla$ by
$$
R^\nabla (X,Y)Z=\nabla_X \nabla_Y Z-\nabla_Y \nabla_X Z-\nabla_{[X,Y]}Z,
\quad X, Y, Z \in \Gamma(TM). 
$$
On a statistical manifold $(M,g,\nabla)$, we often denote $R^{\nabla}$ by $R$, 
$R^{\nabla^g}$ by $R^g$, and $R^{\overline{\nabla}}$ by $\overline{R}$, for short.  
\end{definition}

\vspace{0.5\baselineskip}

\begin{remark}
The following formulas hold for $X, Y, Z, W \in \Gamma(TM)$: 
\begin{eqnarray}
&&
(\nabla g)(Y,Z;X)=(\nabla_X g)(Y,Z)=-2g(K(X,Y),Z), 
\label{CandK} \\
&&
\overline{\nabla}_X Y=\nabla^g_X Y-K(X,Y), \\
&&
g(\overline{R}(X,Y)Z,W)=-g(Z, R(X,Y)W), \\
&&\frac{1}{2}(R+\overline{R})(X,Y)Z=R^g(X,Y)Z+\lbrack K_X, K_Y \rbrack Z. \label{scurvature}
\end{eqnarray}
\end{remark}

\vspace{0.5\baselineskip}

\begin{definition}
A statistical manifold $(M, g, \nabla)$ is said to be {\sl conjugate symmetric}\/
if $R=\overline{R}$ holds. 
\end{definition}

\vspace{0.5\baselineskip}

The properties of the curvature tensor field $R^\nabla$ related to $g$  
are similar to the ones of $R^g$ 
if the considering statistical manifold is conjugate symmetric.  
In particular, we have 
\begin{equation}
g(R^\nabla (Z,W)X,Y)=g(R^\nabla (X,Y)Z,W)
\end{equation}
for $X, Y, Z, W \in \Gamma(TM)$. 

\vspace{0.5\baselineskip}

\begin{definition}
A statistical manifold $(M, g, \nabla)$ is said 
to be {\sl trace-free}\/ 
or to satisfy the {\sl apolarity condition}\/ 
if  $\mathrm{tr}_gK=0$.  
\end{definition}

\vspace{0.5\baselineskip}

The {\sl apolarity}\/ is in the terminology of equiaffine geometry. 
On the other hand, in centroaffine geometry,
$T=T^{(g, \nabla)}=\frac{1}{m}\mathrm{tr}_gK \in \Gamma(TM)$ is often called 
the {\sl Tchebychev vector field}, 
and $\mathcal{T}=\mathcal{T}^{(g, \nabla)} 
=\nabla^g T \in \Gamma(TM^{(1,1)})$
the {\sl Tchebychev operator}. See \cite{MR1317826} for example. 

\vspace{0.5\baselineskip}

\begin{definition}
Let $(M,g,\nabla)$ be a statistical manifold. 
The {\sl curvature interchange}\/  tensor field  
$L=L^{(g,\nabla)} \in\Gamma(TM^{(1,3)})$ is defined by
    $$
    g(L(Z,W)X,Y) = g(R^\nabla (X,Y)Z,W),\quad X,Y,Z,W\in\Gamma(TM).
    $$
We denote $\bar{L}=L^{(g, \overline{\nabla})}$. 
\end{definition}

\vspace{0.5\baselineskip}

\begin{example}
(1) \ 
If a statistical structure $(g,\nabla)$ is conjugate symmetric, 
then $L=R$ holds. 

(2) \ 
Suppose that $R$ is written as 
    \begin{equation}\label{eq:GuassEqn0}
        R(X,Y)Z=g(Y,Z)SX-g(X,Z)SY,\quad X,Y,Z\in\Gamma(TM) 
    \end{equation}
for some $S \in \Gamma(TM^{(1,1)})$ such that $g(SX,Y)=g(X,SY)$. 
Then the curvature interchange tensor field of  $(g,\nabla)$ is given as 
   \begin{equation*}
        L(Z,W)X=-g(X,Z)SW+g(SY,W)Z.
    \end{equation*}
\end{example}

\vspace{0.5\baselineskip}

\begin{remark}
The curvature interchange tensor field $L$ has the following properties 
for $X,Y,Z,W\in\Gamma(TM)$:  
    \begin{eqnarray*}
        &&L(Z,W)Y=-\bar{L}(W,Z)Y, \\
        &&g(L(Z,W)Y,X)=-g(Y,L(Z,W)X), \\
        &&\overline{R}(Y,Z)W=L(Y,W)Z-L(Z,W)Y.
    \end{eqnarray*}
\end{remark}

\vspace{0.5\baselineskip}

\begin{example}\label{geost}
On the Euclidean space $(\mathbb{R}^2,g_0)$, 
define a torsion-free affine connection $\nabla$ as follows.
\begin{equation*}
        \nabla_{\frac{\partial}{\partial x}}\frac{\partial}{\partial x}
        =\frac{\partial}{\partial x},\quad \nabla_{\frac{\partial}{\partial x}}\frac{\partial}{\partial y}=0,\quad \nabla_{\frac{\partial}{\partial y}}\frac{\partial}{\partial y}=\frac{\partial}{\partial y}.
\end{equation*}
The triplet $(\mathbb{R}^2,g_0,\nabla)$ is a statistical manifold. 
It is easy to see that $\nabla$ is flat, 
and accordingly the statistical structure $(g_0,\nabla)$ is conjugate symmetric. 
We have that $\nabla^{g_0}\mathrm{tr}_{g_0}K=0$ while $\mathrm{tr}_{g_0}K\neq0$ on some points.
\end{example}

\vspace{0.5\baselineskip}

\begin{example}\label{Equiaffine Geometry}
Let $(\mathbb{R}^{m+1}, D, \det)$ be a standard equiaffine space. 
By definition, $D$ denotes the standard flat affine connection of $\mathbb{R}^{m+1}$, 
and $\det$ the standard volume form,  which is given as the determinant 
with respect to the standard affine coordinate system with respect to $D$. 
The group $SL(m+1, \mathbb{R}) \ltimes \mathbb{R}^{m+1}$ acts on this space 
preserving these structures. 
Two hypersurfaces $f_1, f_2: M \to \mathbb{R}^{m+1}$ are {\sl equiaffinely congruent}\/  
if there exist $A \in SL(m+1, \mathbb{R})$ and $b \in \mathbb{R}^{m+1}$ such that 
$f_2(x)=A f_1(x)+b$ for all $x \in M$. 

Let $f:M \to \mathbb{R}^{m+1}$ be a locally strongly convex hypersurface 
with the Blaschke normal vector field $\xi \in \Gamma(f^{-1}T\mathbb{R}^{m+1})$. 
By definition, $f$ and $\xi$ satisfy the following properties: 
At each point $x \in M$, we have the decomposition 
$T_{f(x)}\mathbb{R}^{m+1}=(df)_xT_xM \oplus \mathbb{R} \xi_x$.  
According to this decomposition, define $\nabla, h, S$ and $\tau$ by 
\begin{eqnarray*}
&&
D_Xf_*Y=f_*\nabla_XY+h(X,Y)\xi, \\
&&
D_X \xi=-f_*SX+\tau(X)\xi
\end{eqnarray*}
for $X,Y\in \Gamma(TM)$. 
Then 
(1) $\tau$ vanishes, and $h$ is a Riemannian metric on $M$.   
(2) $d \mu_h(X_1, \ldots, X_m)
=\det(f_* X_1, \ldots, f_* X_m, \xi)
$ holds for $X_1, \ldots, X_m \in \Gamma(TM)$. 
We remark that $\xi$ is determined for $f$, 
%up to sign, 
as a unit normal vector field in Euclidean geometry.  
We often call $h$ the {\sl Blaschke metric}, 
and $S$ the {\sl affine shape operator}\/ of $f$, respectively. 

Such a hypersurface $f$ is called an {\sl improper affine hypersphere} \/ if $S=0$, 
and an {\sl affine minimal hypersurface}\/ if $\mathrm{tr}\, S=0$. 

By the Codazzi equation, a well-known integrability condition, 
the above $(h, \nabla)$ is a statistical structure on $M$, 
which is called the statistical structure {\sl equiaffinely induced by}\/ $f$.  
Furthermore, the condition (2) implies that 
$(h, \nabla)$ satisfies the apolarity condition: $\mathrm{tr}_hK=0$. 
We have the same formula to \eqref{eq:GuassEqn0} as the Gauss equation: 
\begin{equation*}
        R^\nabla(X,Y)Z=h(Y,Z)SX-h(X,Z)SY,\quad X,Y,Z\in\Gamma(TM). 
\end{equation*}
For more detail, see \cite{NomizuSasaki} for example.  
\end{example}

\vspace{0.5\baselineskip}

\begin{example}%[\cite{MR2832004}]
\label{pbsp}
    Set $\Omega=\{1,\ldots,n+1\}$, and $\mathcal{S}_n$ as the set of all the positive probability functions on $\Omega$, that is,
    \begin{equation*}
        \mathcal{S}_n=\left\{p:\Omega\to (0,1)\middle|\quad \sum_{\omega=1}^{n+1}p(\omega)=1\right\}.
    \end{equation*}
     The set $\mathcal{S}_n$ is a smooth $n$-dimensional manifold with an atlas consisted of a chart $\phi:\mathcal{S}_n \ni p \mapsto (p(1),\ldots,p(n)) \in \mathbb{R}^{n}$. 
     Denote the coordinate system on $\mathcal{S}_n$ derived from $\phi$ by $(\eta_1,\ldots, \eta_n)$. The Fisher information metric $g^F$ on $\mathcal{S}_n$ is defined as
     \begin{equation*}
     \begin{split}
         g_p^F\left(\frac{\partial}{\partial \eta^i},\frac{\partial}{\partial \eta^j}\right)&=\sum_{\omega=1}^{n+1}p(\omega)\frac{\partial}{\partial \eta^i}\mathrm{ln}(p(\omega))\frac{\partial}{\partial \eta^j}\mathrm{ln}(p(\omega))\\
         &=\frac{\delta_{ij}}{p(i)}+\frac{1}{p(n+1)},\quad p\in\mathcal{S}_n.
     \end{split}
     \end{equation*}
The Riemannian metric $g^F$ has constant curvature $1/4$. 
The dual metric of $g^F$ can be determined by the following formula.
\begin{equation*}
    g_p^F(d\eta^i,d\eta^j)=-p(i)p(j)+\delta_{ij}p(i). 
\end{equation*}

     The flat affine connection $\nabla^{(m)}$ defined by $\nabla^{(m)}\frac{\partial}{\partial \eta^i}=0$ on $\mathcal{S}_n$, is often called the \textit{mixture connection}. The pair $(g^F,\nabla^{(m)})$ is a statistical structure on $\mathcal{S}_n$, and the conjugate connection of $\nabla^{(m)}$ with respect to $g^F$ is denoted by $\nabla^{(e)}$, often called the \textit{exponential connection}.
     
On the statistical manifold $(\mathcal{S}_n,g^F,\nabla^{(e)})$, 
$K=\nabla^{(e)}-\nabla^{g^F}$ satisfies the following formulas 
for each $p\in\mathcal{S}_n$: 
\begin{eqnarray*}
&&
g_p^F\left(
K\left(\frac{\partial}{\partial \eta^i},\frac{\partial}{\partial \eta^j}\right), 
\frac{\partial}{\partial \eta^k}\right)
=-\frac{1}{2}\left(\frac{\delta_{ij}\delta_{jk}}{p(i)^2}-\frac{1}{p(n+1)^2}\right), \\
&&         
\mathrm{tr}_gK_p
=\frac{1}{2}\sum_{i=1}^n\left( (n+1)p(i)-1 \right) \frac{\partial}{\partial \eta^i}_p, \\
&&
\operatorname{div}^g(\mathrm{tr}_gK)(p)
=\frac{1}{4}\left(n^2-1+\sum_{i=1}^{n+1}\frac{1}{p(i)}\right) .
\end{eqnarray*}
\end{example}

\vspace{0.5\baselineskip}

\subsection{Connection Laplacians of statistical manifolds}\label{subsec2.2} %%%%%%%%%%

\begin{definition}
Let $(M, g, \nabla)$ be a statistical manifold, $E \to M$ a vector bundle 
with connection $\nabla^E : \Gamma(E) \to \Gamma(E \otimes T^*M)$. 
We define the operator 
$\Delta^E=\Delta^{(g, \nabla, \nabla^E)} : \Gamma(E) \to \Gamma(E)$ as 
\begin{equation*}
    \Delta^E \xi 
    =\mathrm{tr}_g \{ (X,Y) \mapsto \nabla^E_X\nabla^E_Y \xi -\nabla^E_{\nabla_X Y}\xi \}, 
\end{equation*}
and call it the {\sl statistical connection Laplacian}\/ 
with respect to $(g, \nabla)$ and $\nabla^E$.  
\end{definition}

\vspace{0.5\baselineskip}

In the case where $(g,\nabla)$ is Riemannian, 
it is known as the connection Laplacian or the rough Laplacian with respect to $g$ as well. 

Let  $\langle\cdot,\cdot\rangle$  be a fiber metric on $E$. 
We have a connection $\overline{\nabla}^{\substack{\scalebox{0.3}{\phantom{i}}\\E}}$ on $E$ such that
\begin{equation*}
X\langle\xi,\eta\rangle
=\langle\nabla^E_X\xi,\eta\rangle+\langle\xi,\overline{\nabla}^{\substack{\scalebox{0.3}{\phantom{i}}\\E}}_X\eta\rangle 
\end{equation*}
for $X\in\Gamma(TM)$ and $\xi,\eta\in\Gamma(E)$.
The symbol $\bar{\Delta}^E$ denotes the statistical connection Laplacian 
with respect to $(g, \overline{\nabla})$ and $\overline{\nabla}^{\substack{\scalebox{0.3}{\phantom{i}}\\E}}$. 

\vspace{0.5\baselineskip}

\begin{example}
Let $(M, g, \nabla)$ be a statistical manifold of dimension $m$. 
Take an orthonormal basis $\{e_i\}_{i=1, \ldots,m}$ with respect to $g$. 

(1) \ 
For a trivial bundle $E=M \times \mathbb{R}$, 
the statistical connection Laplacian $\Delta=\Delta^E: C^\infty(M) \to C^\infty(M)$ 
is given as 
\begin{eqnarray}
\Delta f 
&=& \sum_{i=1}^m \{ e_i(e_i f)-(\nabla_{e_i}e_i)f\} 
\nonumber \\
&=& \sum_{i=1}^m \{ e_i(e_i f)-(\nabla^g_{e_i}e_i)f\} 
- \sum_{i=1}^m K(e_i,e_i) f \nonumber \\
&=& \Delta_g f -(\mathrm{tr}_g K) f  \label{Lapf}
\end{eqnarray}
for $f \in C^\infty(M)$, 
where $\Delta_g$ is the Laplacian with respect to $g$, 
which is the standard tool in Riemannian geometry.  
Furthermore, we have 
$
\bar{\Delta} f=\Delta_g f+(\mathrm{tr}_g K) f
$.

(2) \ 
Let $N$ be a manifold with an affine connection $\nabla^N$. 
For a map $u:M \to N$, 
we denote by $\nabla^u$ 
the connection on $u^{-1}TN$ induced from $\nabla^N$ by $u$: 
$
\nabla^u_X U=\nabla^N_{u_* X}U \in \Gamma(u^{-1}TN)
$ for 
$X \in \Gamma(TM)$ and $U \in \Gamma(u^{-1}TN)$. 
The symbol $\Delta^u =\Delta^{u^{-1}TN}: \Gamma(u^{-1}TN) \to \Gamma(u^{-1}TN)$ 
denotes the statistical connection Laplacian 
with respect to $(g, {\nabla})$ and $\nabla^u$. 
\end{example}

\vspace{0.5\baselineskip}

We now review divergences and Green's formula for later use. 

For a volume form $\theta$ on $M$, the $\theta$-divergence $\operatorname{div}^{\theta}:\Gamma(TM)\to C^{\infty}(M)$ is defined by
\begin{equation*}
\operatorname{div}^{\theta}X\theta=\mathcal{L}_X\theta,\quad X\in\Gamma(TM), 
\end{equation*}
where $\mathcal{L}$ is the Lie derivative on $M$. 
On the other hand, for an affine connection $\nabla$ on $M$, 
the $\nabla$-divergence $\operatorname{div}^{\nabla}$ is defined by
\begin{equation*}
\operatorname{div}^{\nabla}X=\mathrm{tr}\nabla X,\quad X\in\Gamma(TM).
\end{equation*}
If $\theta$ is parallel with respect to $\nabla$, that is, $\nabla\theta=0$, then
\begin{equation*}    
\operatorname{div}^{\theta}=\operatorname{div}^{\nabla}.
\end{equation*}
For a Riemannian metric $g$, 
we will denote the $\nabla^g$-divergence by $\operatorname{div}^g$. 

The statistical Laplacian has the formula of
\begin{equation*}
    \Delta f=\operatorname{div}^{\overline{\nabla}}(\mathrm{grad}_gf), 
\end{equation*}
where the gradient vector field $\mathrm{grad}_gf$ of $f$ with respect to $g$ is defined by 
$g(\mathrm{grad}_g f, X)=df(X)$ 
for $X \in \Gamma(TM)$. 
%$\mathrm{grad}_gf=\sum_{i=1}^ndf(e_i)e_i$, 

The following facts are well known.  
See \cite{KobayashiNomizu1} for example. 

\vspace{0.5\baselineskip}

\begin{proposition}%[Green's formula ]
{\rm (1)}\ For any $f \in C^\infty(M)$ and $X \in \Gamma(TM)$, we have 
\begin{equation}\label{divfor}
\operatorname{div}^{\nabla}(f X)
=Xf+f\operatorname{div}^{\nabla}X. 
\end{equation}
    
{\rm (2)} (Green's formula) \ 
If $M$ is compact, for any $X\in\Gamma(TM)$ we have
\begin{equation*}        
\int_{M}\operatorname{div}^{\theta}X\theta=0.
\end{equation*}
Besides, if $X \in \Gamma(TM)$ satisfies $\operatorname{div}^\theta X=0$, 
then for any $f \in C^\infty(M)$ we have 
\begin{equation*}   
\int_{M}\, Xf \, \theta=0.
\end{equation*}
\end{proposition}

\vspace{0.5\baselineskip}

Using these facts, we have the following formula which is well-known 
in the case where a statistical manifold is Riemannian. 

\vspace{0.5\baselineskip}

\begin{proposition}\label{lapfor}
Let $(M, g, \nabla)$ be a compact statistical manifold. 
Let $E$ be a vector bundle over $M$ 
with a connection $\nabla^E$ and a fiber metric $\langle \cdot, \cdot \rangle$. 
The formula 
    \begin{equation}\label{lapfor1}        
    \int_{M}\langle\Delta^E\xi,\eta\rangle d\mu_g
=\int_{M}\langle\xi,\Bar{\Delta}^E\eta\rangle d\mu_g
    +\int_{M} \operatorname{div}^g(\mathrm{tr}_gK) \langle\xi, \eta\rangle d\mu_g
    \end{equation}
    holds for any $\xi,\eta\in\Gamma(E)$.
\end{proposition}

\begin{proof}
We set vector fields as 
\begin{equation*}
X=\sum_{i=1}^m\langle\nabla_{e_i}^E\xi,\eta\rangle e_i, \quad 
Y=\sum_{i=1}^m\langle\xi,\overline{\nabla}^{\substack{\scalebox{0.3}{\phantom{i}}\\E}}_{e_i}\eta\rangle e_i, \quad
Z=\langle\xi,\eta\rangle \mathrm{tr}_gK, 
\end{equation*}
where $\{ e_i \}_{i=1, \ldots, m}$ is an orthonormal basis with respect to $g$. 
Then, we have
\begin{equation}\label{lapfor0}
\operatorname{div}^g(X-Y-Z)
=
 \langle\Delta^E\xi,\eta\rangle
-\langle\xi,\Bar{\Delta}^E\eta\rangle
-\operatorname{div}^g(\mathrm{tr}_gK) \langle\xi, \eta\rangle , 
\end{equation}
and Green's formula implies \eqref{lapfor1}. 
The calculation is as follows. 
\begin{eqnarray}
\operatorname{div}^g X
&=&\sum_{i=1}^m g(\nabla^g_{e_i}X,e_i) \nonumber\\
&=&\sum_{i=1}^m e_i\langle\nabla_{e_i}^E\xi,\eta\rangle+\sum_{i,j=1}^m\langle\nabla_{e_j}^E\xi,\eta\rangle g(\nabla^g_{e_i}e_j,e_i) \nonumber\\
&=&\sum_{i=1}^m
\{
\langle\nabla^E_{e_i}\nabla_{e_i}^E\xi,\eta\rangle
+\langle\nabla_{e_i}^E\xi,\overline{\nabla}^{\substack{\scalebox{0.3}{\phantom{i}}\\E}}_{e_i}\eta\rangle
\} \nonumber\\
&& \quad 
-\sum_{i,j=1}^m
\sum_{i,j=1}^m\langle\nabla_{e_j}^E\xi,\eta\rangle g(\nabla^g_{e_i}e_i,e_j) \nonumber\\
&=&\sum_{i=1}^m
\{
\langle\nabla^E_{e_i}\nabla_{e_i}^E\xi,\eta\rangle
+\langle\nabla_{e_i}^E\xi,\overline{\nabla}^{\substack{\scalebox{0.3}{\phantom{i}}\\E}}_{e_i}\eta\rangle
\}  
- \sum_{i=1}^m\langle\nabla_{\nabla^g_{e_i}e_i}^E\xi,\eta\rangle \nonumber\\
&=&\sum_{i=1}^m
\{
\langle\nabla^E_{e_i}\nabla_{e_i}^E\xi,\eta\rangle
+\langle\nabla_{e_i}^E\xi,\overline{\nabla}^{\substack{\scalebox{0.3}{\phantom{i}}\\E}}_{e_i}\eta\rangle
\} \nonumber\\
&& \quad 
- \sum_{i=1}^m\langle\nabla_{\nabla_{e_i}e_i}^E\xi,\eta\rangle
+\langle\nabla_{\mathrm{tr}_gK}^E\xi,\eta\rangle  \nonumber \\
&=&
\langle \Delta^E\xi,\eta\rangle+\langle\nabla^E_{\mathrm{tr}_gK}\xi,\eta\rangle
+\sum_{i=1}^m\langle\nabla_{e_i}^E\xi,\overline{\nabla}^{\substack{\scalebox{0.3}{\phantom{i}}\\E}}_{e_i}\eta\rangle. 
\label{divX}
\end{eqnarray}
    
Similarly, we have
    \begin{equation}\label{divY}    
    \operatorname{div}^gY
    =\langle\xi, 
    \Bar{\Delta}^E\eta\rangle-\langle\xi,\overline{\nabla}^{\substack{\scalebox{0.3}{\phantom{i}}\\E}}_{\mathrm{tr}_gK}\eta\rangle
    +\sum_{i=1}^m\langle\nabla^E_{e_i}\xi,\overline{\nabla}^{\substack{\scalebox{0.3}{\phantom{i}}\\E}}_{e_i}\eta\rangle.
    \end{equation}
    
Lastly, by \eqref{divfor} we have 
\begin{equation}\label{divZ}
\begin{split}
\operatorname{div}^gZ
&=\mathrm{tr}_gK\langle\xi,\eta\rangle
+\langle\xi,\eta\rangle\operatorname{div}^g(\mathrm{tr}_gK)\\
&=\langle\nabla^E_{\mathrm{tr}_gK}\xi,\eta\rangle
    +\langle\xi,\overline{\nabla}^{\substack{\scalebox{0.3}{\phantom{i}}\\E}}_{\mathrm{tr}_gK}\eta\rangle
    +\langle\xi,\eta\rangle\operatorname{div}^g(\mathrm{tr}_gK).
\end{split}
\end{equation}
By combining $(\ref{divX})$, $(\ref{divY})$, and $(\ref{divZ})$, 
we obtain the desired $(\ref{lapfor0})$.
\end{proof}

\vspace{0.5\baselineskip}

Let $E$ be a vector bundle over a statistical manifold $(M, g, \nabla)$ 
with a fiber metric $\langle \cdot , \cdot \rangle$. 
For a connection $\nabla^E: \Gamma(E) \to \Gamma(E \otimes T^*M)$, 
the {\sl adjoint}\/  $(\nabla^E)^* :  \Gamma(E \otimes T^*M) \to \Gamma(E)$  is defined by
\begin{equation*}
    (\nabla^E)^*\mathcal{E}
    =-\sum_{i=1}^m\left\{ \overline{\nabla}^{\substack{\scalebox{0.3}{\phantom{i}}\\E}}_{e_i}(\mathcal{E}(e_i))-\mathcal{E}(\overline{\nabla}_{e_i}{e_i})\right\} 
\end{equation*}
for $\mathcal{E}\in\Gamma(E \otimes T^*M)$. 

We remark that the fiber metric of $E \otimes T^*M$ 
is naturally derived from $g$ and $\langle \cdot, \cdot \rangle$, 
and we denote it by $\langle \cdot, \cdot \rangle$ again. 
Using it, we write the adjointness of $\nabla^E$ and $(\nabla^E)^*$ as follows. 

\vspace{0.5\baselineskip}

\begin{proposition}[\cite{MR3422914}]\label{adjoint-ness}
Let $(M,g,\nabla), E, \langle \cdot, \cdot \rangle, \nabla^E$ be as above.  

{\rm (1)} \ 
Suppose that $M$ is compact and there is a $\nabla$-parallel volume form $\theta$ on $M$. 
The formula  
    \begin{equation*}     
    \int_{M}\langle\mathcal{E},\nabla^E\eta\rangle\theta 
    = \int_{M}\langle(\nabla^E)^*\mathcal{E},\eta\rangle\theta
    \end{equation*}
holds for any $\eta\in\Gamma(E)$ and $\mathcal{E}\in\Gamma(E\otimes T^*M)$. 

{\rm (2)} \ 
We have 
\begin{equation*}
    \Delta^E\xi=-(\overline{\nabla}^{\substack{\scalebox{0.3}{\phantom{i}}\\E}})^*\nabla^E\xi,\quad\xi\in\Gamma(E).
\end{equation*}    
\end{proposition}

\vspace{0.5\baselineskip}

We remark that our sign convention is different from \cite{MR3422914} for example. 
Moreover, the statistical connection Laplacian $\Delta^E$ is apart from $\triangle$ 
in \cite{JiangTavakoliZhao} as well. 
 
\vspace{0.5\baselineskip}

\section{The first variation formula of the statistical bi-energy}\label{sec3} %%%%%%%%%%

Let $u:M\to N$ be a smooth map from a compact statistical manifold $(M,g,\nabla^M)$ to a statistical manifold $(N,h,\nabla^N)$,  
and let $\tau_2(u) \in \Gamma(u^{-1}TN)$ be the statistical bi-tension field 
for $u$ given in \eqref{eq:bitension}. 
A smooth variation $F=\{u_t\}_{t\in (-\epsilon, \epsilon)}$ of $u$ 
is a smooth map $F:M\times (-\epsilon, \epsilon) \to N$ 
such that $F(\cdot, 0)=u_0=u$. 
The smooth map $F$ yields a vector field $V\in\Gamma(u^{-1}TN)$ by
\begin{equation*}
    V(x)=\left. \frac{d}{dt} \right|_{t=0} u_t(x), \quad x \in M, 
\end{equation*} 
called the variation vector field of $F$. 
Conversely, it is known that if we take any $V\in\Gamma(u^{-1}TN)$, 
there exists a smooth variation of $u$ that generates $V$.
In this section, we give the first variation formula 
of the statistical bi-energy $E_2$ 
for smooth maps between statistical manifolds, defined in \eqref{eq:bienergy}. 

\vspace{0.5\baselineskip}

\begin{theorem}\label{FVF}
Let $u:M\to N$ be a smooth map from a compact statistical manifold $(M,g,\nabla^M)$ 
to a statistical manifold $(N,h,\nabla^N)$. 
For an arbitrary smooth variation 
$F=\{u_t\}_{t\in (-\epsilon, \epsilon)}$ generating $V$, 
the first variation formula is 
\begin{equation}\label{varfor}
    \left.\displaystyle\frac{d}{dt}\right|_{t=0}E_2(u_t)
    =\int_{M}\left\langle V,\tau_2(u)\right\rangle d\mu_g.
\end{equation}
\end{theorem}

\begin{proof}
We take the standard Euclidean metric and its Levi-Civita connection 
on $(-\epsilon,\epsilon) \subset \mathbb{R}$. 
On the product statistical manifold $B = M\times (-\epsilon,\epsilon)$, 
we denote the induced statistical connection by $\nabla^B$. 
Identifying $X\in\Gamma(TM)$ 
with $(X,0)\in\Gamma(TB)\simeq\Gamma(TM\oplus T(-\epsilon,\epsilon))$, 
we have 
    \begin{equation}    \label{conF}
        \nabla^B_X Y 
        = \nabla^M_X Y,\quad \nabla^B_X \frac{\partial}{\partial t}
        =\nabla^B_{\frac{\partial}{\partial t}} X=0,\quad
        \nabla^B_{\frac{\partial}{\partial t}}\frac{\partial}{\partial t}=0
    \end{equation}
for any $X,Y\in\Gamma(TM)$.

Let $\nabla^F$ be the connection on $\Gamma(F^{-1}TN)$ 
induced from $\nabla^N$, and set $E=F^{-1}TN\otimes T^*B$. 
We remark that $\nabla^E$ is given by 
    \begin{equation*}
        (\nabla^E_X \Phi)(Y)=\nabla^F_X (\Phi Y)-\Phi\nabla^B_X Y\in\Gamma(F^*TN)
    \end{equation*}
for $X,Y\in\Gamma(TB)$ and $\Phi\in\Gamma(E)$. 
The curvature tensor field $R^E$ of $\nabla^E$ is written by 
\begin{equation}\label{curvtens0}
    R^E(X,Y)\Phi(Z) 
    = \left((\nabla^E)^2\Phi\right)(Z;Y;X)-\left((\nabla^E)^2\Phi\right)(Z;X;Y) 
\end{equation}
for $X,Y,Z\in\Gamma(TB)$ and $\Phi\in\Gamma(E)$. 
Besides, it holds that 
\begin{equation}\label{curvtens}
    R^E(X,Y)dF(Z)=R^N\left(F_*X,F_*Y\right)F_*Z-F_*\left(R^B(X,Y)Z\right)
\end{equation}
for $X,Y,Z\in\Gamma(TB)$, 
where $R^B$ and $R^N$ are the curvature tensor fields of $\nabla^B$ 
and $\nabla^N$, respectively. 

Using \eqref{CandK}, we have 
\begin{equation*}
\begin{split}
    \displaystyle\frac{d}{dt}E_2(u_t)
    &=\displaystyle
    \frac{1}{2} \int_{M}\frac{\partial}{\partial t} ||\tau(u_t)||_h^2 d\mu_g\\
    &=\int_{M}h\left(\nabla^F_{\frac{\partial}{\partial t}}\tau(u_t),\tau(u_t)\right)d\mu_g
    +\displaystyle\frac{1}{2}\int_{M}(\nabla^N_{F_*\frac{\partial}{\partial t}}h)\Bigl(\tau(u_t),\tau(u_t)\Bigr)d\mu_g\\
%    &=\int_{M}h\left(\nabla^F_{\frac{\partial}{\partial t}}\tau(u_t),\tau(u_t)\right)d\mu_g+\displaystyle\frac{1}{2}\int_{M}(\nabla^N_{\tau(u_t)}h)\left(\tau(u_t),F_*\frac{\partial}{\partial t}\right)d\mu_g\\
    &=\int_{M}h\left(\nabla^F_{\frac{\partial}{\partial t}}\tau(u_t),\tau(u_t)\right)d\mu_g
    -\int_{M}h\left(K^N\bigl(\tau(u_t),\tau(u_t)\bigr),F_*\frac{\partial}{\partial t}\right)d\mu_g.
\end{split}
\end{equation*}
Taking an orthonormal basis $\{ e_i \}_{i=1, \ldots, m}$ with respect to $g$, 
we have, by \eqref{curvtens0} and by $(\ref{curvtens})$ and $(\ref{conF})$,
\begin{equation}\label{covtau}
    \begin{split}
        \nabla^F_{\frac{\partial}{\partial t}}\tau(u_t)
        &=\sum_{i=1}^m\nabla^F_{\frac{\partial}{\partial t}}
        \left((\nabla^E_{e_i}dF)(e_i)\right)\\
        &=\sum_{i=1}^m\left((\nabla^E)^2dF\right)
        \left(e_i;e_i;\frac{\partial}{\partial t}\right)\\
        &=\sum_{i=1}^m\left((\nabla^E\right)^2dF)
        \left(e_i;\frac{\partial}{\partial t};e_i\right)
        +\sum_{i=1}^m\left(R^E\left(\frac{\partial}{\partial t},e_i\right)dF\right)(e_i) \\
        &=\sum_{i=1}^m\left((\nabla^E\right)^2dF)
        \left(e_i;\frac{\partial}{\partial t};e_i\right)
    +\sum_{i=1}^m{R^N}\left(F_*\frac{\partial}{\partial t},F_*(e_i)\right)F_*(e_i) \\
        &=\Delta^{u_t}\left(F_*\frac{\partial}{\partial t}\right)
        +\sum_{i=1}^m{R^N}\left(F_*\frac{\partial}{\partial t}, 
        F_*(e_i)\right)F_*(e_i).
    \end{split}
\end{equation}
Indeed, the last equality is obtained 
by $(\ref{conF})$ as follows:  
\begin{equation*}
    \begin{split}
        \sum_{i=1}^m\left((\nabla^E)^2dF\right)
        \left(e_i;\frac{\partial}{\partial t};e_i\right)
        &=\sum_{i=1}^m\left(\nabla^F_{e_i}
        \left((\nabla^E_{\frac{\partial}{\partial t}}dF)(e_i)\right)
        -(\nabla^E_{\frac{\partial}{\partial t}}dF)
        (\nabla^B_{e_i}{e_i})\right)\\
&=\sum_{i=1}^m\left(\nabla^{u_t}_{e_i}\nabla^{u_t}_{e_i}
\left(F_*\frac{\partial}{\partial t}\right)
-\nabla^{u_t}_{\nabla^M_{e_i}{e_i}}
\left(F_*\frac{\partial}{\partial t}\right)\right)\\
        &=\Delta^{u_t}\left(F_*\frac{\partial}{\partial t}\right) . 
    \end{split}
\end{equation*}

By Proposition $\ref{lapfor}$, we have 
\begin{equation*}
    \begin{split}
    &\int_{M}h\left(\Delta^{u_t}
    \left(F_*\frac{\partial}{\partial t}\right),\tau(u_t)\right)d\mu_g\\
    &=\int_{M}h\left(F_*\frac{\partial}{\partial t},\Bar{\Delta}^{u_t}\tau(u_t)
    +\operatorname{div}^g(\mathrm{tr}_gK)\tau(u_t)\right)d\mu_g.
    \end{split}
\end{equation*}
By combining them and remembering 
$
V=\left.F_*\frac{\partial}{\partial t}\right|_{t=0}
$, 
we have the first variation formula: 
\begin{eqnarray*}
\left.\displaystyle\frac{d}{dt}\right|_{t=0}E_2(u_t) 
&=&\int_{M} \langle \, V, \,  \bar{\Delta}^u\tau(u)+\operatorname{div}^g
(\mathrm{tr}_g K^M)\tau(u)  \\
&& \qquad \qquad 
-\sum_{i=1}^mL^N(u_*e_i,\tau(u)) u_*e_i 
-K^N(\tau(u),\tau(u) ) \,     \rangle d\mu_g.
\end{eqnarray*}
\end{proof}

\vspace{0.5\baselineskip}

\section{Examples of statistical biharmonic maps}\label{sec4} %%%%%%%%%%

As stated in \S 1, we call a map $u$ statistical biharmonic 
if the statistical bi-tension field of $u$ vanishes identically. 
Under suitable assumptions, we will rewrite the statistical bi-tension field.  

\vspace{0.5\baselineskip}

%\begin{remark}
Let $u: (M,g, \nabla^M) \to (N, h, \nabla^N)$ be a smooth map 
between statistical manifolds. 
Suppose that $M$ is trace-free and $N$ is conjugate symmetric. 
Then the statistical bi-tension field of $u$ is given as 
\begin{equation}\label{bt41}
\tau_2(u)
=\bar{\Delta}^u \tau(u) 
-\sum_{i=1}^mR^N(u_*e_i, \tau(u))u_*e_i-K^N(\tau(u), \tau(u)),  
\end{equation}
where $\{e_j\}_{j=1, \ldots, m}$ is an orthonormal basis 
with respect to $g$. 
%\end{remark}

\vspace{0.5\baselineskip}

\begin{example}
(1) \ Let $c: (M,g, \nabla^M) \to (N, h, \nabla^N)$ be a smooth map 
between statistical manifolds. 
Suppose that $(M, g, \nabla^M)=\left((a,b), g_0, \nabla^{g_0}\right)$, 
where $g_0$ is the Euclidean metric on $(a,b)\subset\mathbb{R}$. 
Since $\tau(c)=\nabla^N_{\dot{c}}\dot{c}$, where $\dot{c}=c_*\dfrac{d}{dt}$,  
the statistical bi-tension field of $c$ is given as 
    \begin{equation*}
    \tau_2(c)
    =\overline{\nabla}^{\substack{\scalebox{0.3}{\phantom{i}}\\N}}_{\dot{c}}\overline{\nabla}^{\substack{\scalebox{0.3}{\phantom{i}}\\N}}_{\dot{c}}\nabla^N_{\dot{c}}\dot{c}
    -L^N(\dot{c},\nabla^N_{\dot{c}}\dot{c})\dot{c}
    -K^N\left(\nabla^N_{\dot{c}}\dot{c},\nabla^N_{\dot{c}}\dot{c}\right). 
    \end{equation*} 
Accordingly, any $\nabla^N$-geodesic $c$ is a statistical biharmonic map. 

 (2) \ 
Take the statistical manifold $(\mathbb{R}^2,g_0,\nabla)$ 
in Example $\ref{geost}$ as $(N, h, \nabla^N)$ in (1).  
For any real numbers $a,b,c$, and $d$, 
the curve $c(t)=(at+c,bt+d)\in\mathbb{R}^2$, $t\in\mathbb{R}$,  
satisfies $\tau_2(c)=0$. 
Here, if $a\neq0$ or $b\neq0$, then $\tau(c)\neq0$.
\end{example}

\vspace{0.5\baselineskip}

\begin{example}
(1) \     
Let $f: (M,g, \nabla^M) \to (N, h, \nabla^N)$ be a smooth map 
between statistical manifolds. 
Suppose that $(N, h, \nabla^N)=(\mathbb{R}, g_0, \nabla^{g_0})$.  
Since $\tau(f)=\Delta f$,  the statistical bi-tension field of $f$ 
is given as 
\begin{equation*}
        \tau_2(f)=\bar{\Delta}\Delta f
        +\operatorname{div}^g(\mathrm{tr}_gK)\Delta f, 
\end{equation*}
where $\Delta=\Delta^{(g,\nabla,\nabla^{g_0})}$ and $\bar{\Delta}=\Delta^{(g,\overline{\nabla},\nabla^{g_0})}$.

(2) \     
If we take the statistical manifold $(\mathbb{R}^2,g_0,\nabla)$ in Example $\ref{geost}$, 
we have 
    \begin{equation*}
        \tau_2(f)
        =\bar{\Delta}\Delta f
        =\left\{\left(\frac{\partial^2}{\partial x^2}
        +\frac{\partial^2}{\partial y^2}\right)^2
        -\left(\frac{\partial}{\partial x}
        +\frac{\partial}{\partial y}\right)^2\right\}f. 
    \end{equation*}
The function $f(x,y)=\sinh{x}+\cosh{y},~(x,y)\in\mathbb{R}^2$, 
satisfies $\tau_2(f)=0$, while $\tau(f)$ is not globally equal to $0$. 
\end{example}

\vspace{0.5\baselineskip}

\begin{lemma}
\label{Lapcon}
Suppose that $(M,g, \nabla)$ is a compact statistical manifold with $\operatorname{div}^g(\mathrm{tr}_gK)=0$. 
For $f\in C^{\infty}(M)$,  if $\Delta f$ is constant, 
then so is $f$. 
\end{lemma}

\begin{proof}
We put $X\in\Gamma(TM)$ as $X=\Delta f \, \mathrm{grad}_gf$. 
By \eqref{Lapf} and \eqref{divfor}, 
we have
    \begin{equation*}
        \begin{split}
            \operatorname{div}^gX&=\Delta f\Delta_g f\\
            &=(\Delta f)^2+\Delta f\mathrm{tr}_gKf\\
            &=(\Delta f)^2+\operatorname{div}^g(f\Delta f\mathrm{tr}_gK).
        \end{split}
    \end{equation*}
By Green's formula, we have
    \begin{equation*}
        0=\int_{M}(\Delta f)^2 d\mu_g,
    \end{equation*}
thus we obtain $\Delta f=0$. 

We remark that 
$
\Delta_g f^2 = 2 f \Delta_g f+2\| df \|_g^2
$. 
By \eqref{Lapf} and Green's formula again, we have
    \begin{equation*}
        \begin{split}
            0&=\int_{M}f\Delta fd\mu_g\\
            &=\int_{M}f\Delta_g fd\mu_g-\int_{M}f\mathrm{tr}_gKfd\mu_g\\
            &=-\int_{M}\|df\|_g^2 d\mu_g - \frac{1}{2}\int_M\mathrm{tr}_gKf^2d\mu_g\\
            &=-\int_{M}\|df\|_g^2 d\mu_g,
        \end{split}
    \end{equation*}
    therefore $df=0$ on $M$.
\end{proof}

\begin{proposition}
Suppose that $(M,g, \nabla)$ is a compact statistical manifold 
with  $\operatorname{div}^g(\mathrm{tr}_gK)=0$.  
A function $f: (M,g,\nabla) \to (\mathbb{R}, g_0, \nabla^{g_0})$ 
is statistical biharmonic if and only if $f$ is constant.
\end{proposition}

\begin{proof}
Since $\operatorname{div}^g(\mathrm{tr}_gK)=0$, we have
    \begin{equation*}
    \tau_2(f)=\bar{\Delta}\Delta f=\Delta_g\Delta f+\mathrm{tr}_gK\Delta f.
    \end{equation*}
By integration, we obtain
    \begin{equation*}
    \begin{split}
        0
        &=\int_M (\Delta_g\Delta f+\mathrm{tr}_gK\Delta f)\Delta f d\mu_g\\
        &=\int_{M} \left\{ 
        -\|d\Delta f\|_g^2+\frac{1}{2} \Delta_g(\Delta f)^2 
        +\frac{1}{2} \mathrm{tr}_g K (\Delta f)^2
        \right\} d\mu_g\\
        &=-\int_{M} \|d\Delta f\|_g^2d\mu_g,
    \end{split}
    \end{equation*}
hence, $\Delta f$ is constant. 
From Lemma $\ref{Lapcon}$, we conclude that $f$ is a constant function.
\end{proof}

\vspace{0.5\baselineskip}

\begin{theorem}    \label{affineimmersion}
Let $f: M \to \mathbb{R}^{m+1}$ be a locally strongly convex hypersurface 
with affine shape operator $S$,  
and $(h, \nabla)$ the statistical structure equiaffinely induced by $f$ 
as in Example \ref{Equiaffine Geometry}. 
Consider $f$ as a map from a statistical manifold $(M,h,\nabla)$ 
to a statistical manifold $(\mathbb{R}^{m+1},g_0, \nabla^{g_0})$ 
for the Euclidean metric $g_0$ 
compatible with the equiaffine structure of $\mathbb{R}^{m+1}$. 
Then, $\tau(f)$ does not vanish. 
Moreover, $f$ is statistical biharmonic if and only if 
(1) $\mathrm{tr}\, S=0$ and (2) $\mathrm{tr}_h \nabla S=0$ hold. 

In particular, 
an improper affine hypersphere is a statistical biharmonic map. 
\end{theorem}

\vspace{0.5\baselineskip}
 
\begin{proof}
Let denote $\nabla^{g_0}$ by $D$ for short.   
For $f:(M,h,\nabla)\to (\mathbb{R}^{m+1},g_0,D)$, 
\begin{equation*}
        \tau(f)
        =\mathrm{tr}_h\{(X,Y)\mapsto D_Xf_*Y-f_*\nabla_XY\}=\mathrm{tr}_h\{(X,Y)\mapsto h(X,Y)\xi \}
        =m\xi, 
\end{equation*}
which does not vanish. 
On the other hand, since $\mathrm{tr}_hK$ vanishes, 
by \eqref{bt41} the statistical bi-tension field $\tau_2(f)$ of $f$ is obtained as 
\begin{equation*}
\begin{split}
    \frac{1}{m}\tau_2(f)
    &=\sum_{i=1}^m
    \left(D_{e_i}D_{e_i}\xi-D_{\overline{\nabla}_{e_i}e_i}\xi\right) \\   
    &=\sum_{i=1}^m
    \left(-D_{e_i}f_*Se_i+f_*S\overline{\nabla}_{e_i}e_i\right) \\
    &=\sum_{i=1}^m
    \left(-f_*\nabla_{e_i}Se_i-h(e_i,Se_i)\xi+f_*S\nabla_{e_i}e_i\right) \\ 
    &=-f_*\left(\mathrm{tr}_h\nabla S\right)-\mathrm{tr}S\, \xi. 
\end{split}
\end{equation*}
Therefore, $\tau_2(f)$ vanishes if and only if (1) and (2) hold. 
\end{proof}

\vspace{0.5\baselineskip}

\section{Reduction properties of statistical biharmonic maps}\label{sec5} %%%%%%%%%%%

Let $(M,g,\nabla^M)$, $(N,h,\nabla^N)$ be statistical manifolds, 
and $u:M\to N$ a smooth map with $\tau(u) \in \Gamma(u^{-1}TN)$,  
which is defined for $u$ by using $(g, \nabla^M)$ and $\nabla^N$  
as in \eqref{tension}.  

Remark that if $(g, \nabla^M)$ and $(h, \nabla^N)$ are Riemannian, 
$\tau(u)$ is called the {\sl tension field}\/ of $u$, 
and a map $u:M\to N$ satisfying $\tau(u)=0$ is classically called a {\sl harmonic map}. 
In the case where $\nabla^M$ is flat, and $(h,\nabla^N)$ is Riemannian, 
a map $u:M\to N$ satisfying $\tau(u)=0$ is studied as an {\sl affine harmonic map}\/ in \cite{Jostcimcir}. 

\vspace{0.5\baselineskip}

\begin{remark}
Let  $(M,g,\nabla^M)$ a compact statistical manifold 
equipped with a $\overline{\nabla}^{\substack{\scalebox{0.3}{\phantom{i}}\\M}}$-parallel volume form $\theta$, 
and $(N,h,\nabla^N)$ be a Riemannian one.  
By Proposition \ref{adjoint-ness}, 
the equation $\tau(u)=0$ holds 
if and only if the map $u:M\to N$ is a critical point 
of the following functional:     
    \begin{equation*}
        E(u)=\frac{1}{2}\int_{M}\|du\|_{g,h}^2\, \theta .
    \end{equation*}
\end{remark}

\vspace{0.5\baselineskip}

In the following, we observe some settings 
in which $\tau_2(u)=0$ implies $\tau(u)=0$. 

\vspace{0.5\baselineskip}

\begin{lemma}\label{TauPar}
Let $(M,g,\nabla^M)$, $(N,h,\nabla^N)$ be statistical manifolds. 
Assume that $M$ is compact, 
and that there exists a $\overline{\nabla}^{\substack{\scalebox{0.3}{\phantom{i}}\\M}}$-parallel volume form $\theta$ on $M$. 
For a map $u:M \to N$, 
if $\overline{\nabla}^{\substack{\scalebox{0.3}{\phantom{i}}\\N}} \tau(u)$ vanishes, then so does $\tau(u)$. 
\end{lemma}

\begin{proof}
We put $X\in\Gamma(TM)$ as
    \begin{equation*}
        X=\sum_{i=1}^mh\bigl(u_*e_i,\tau(u)\bigr)e_i, 
    \end{equation*}
and obtain
    \begin{equation*}
    \begin{split}
    \operatorname{div}^{\theta}X
    &=\sum_{i=1}^mg(\overline{\nabla}^{\substack{\scalebox{0.3}{\phantom{i}}\\M}}_{e_i}X,e_i)\\
    &=\sum_{i=1}^m e_ih\bigl(u_*e_i,\tau(u)\bigr)
    +\sum_{i,j=1}^mh\left(u_*e_j,\tau(u)\right)
    g\left(\overline{\nabla}^{\substack{\scalebox{0.3}{\phantom{i}}\\M}}_{e_i}{e_j},e_i\right)\\
    &=\sum_{i=1}^mh\bigl(\nabla^N_{e_i}u_*e_i,\tau(u)\bigr)
    -\sum_{i=1}^mh\left(u_*(\nabla^M_{e_i}{e_i}),\tau(u)\right)\\
    &=h(\tau(u),\tau(u)).
    \end{split}
    \end{equation*}
Since $\theta$ is $\overline{\nabla}^{\substack{\scalebox{0.3}{\phantom{i}}\\M}}$-parallel,  
by Green's formula for the integration by $\theta$, we have 
    \begin{equation*}
        0=\int_{M}  \| \tau(u) \|_h^2 \, \theta,
    \end{equation*}
which concludes $\tau(u)=0$.
\end{proof}

\vspace{0.5\baselineskip}

\begin{proposition}
Let $(M,g,\nabla^M)$, $(N,h,\nabla^N)$ be statistical manifolds. For $(M,g,\nabla^M)$, assume that $M$ is compact, there exists a $\overline{\nabla}^{\substack{\scalebox{0.3}{\phantom{i}}\\M}}$-parallel volume form $\theta$ on $M$, and that $\operatorname{div}^g(\mathrm{tr}_gK^M)=0$ on $M$. 
For $(N,h,\nabla^N)$, assume that $(N,h,\nabla^N)$ is conjugate symmetric, 
and that the $U^N$-sectional curvature is non-positive on $(N,h)$, 
that is, for unit tangent vectors $X,Y$ on $(N,h)$,
    \begin{equation*}
        h\left(U^N(X,Y)Y,X\right)\leq 0, 
    \end{equation*}
where $U^N=2R^h-R^N$. 
If a statistical biharmonic map $u:(M,g,\nabla^M)\to(N,h,\nabla^N)$ 
satisfies $K^N\bigl(\tau(u),\tau(u)\bigr)=0$, then $\tau(u)=0$.
\end{proposition}

\vspace{0.5\baselineskip}

The above $U^N$ is found as an 
$
U^N=2R^h-(R^N+\overline{R}^{\substack{\scalebox{0.3}{\phantom{i}}\\N}})/2 \in \Gamma(TN^{(1,3)})
$ 
in \cite{FuruhataHasegawaSatoh2022}, 
in which the sectional curvature is defined by using $U^N$ 
even if $(h, \nabla^N)$ is not conjugate symmetric. 
This proposition is proved in \cite{GuoyingJiang} in the case where
$(g, \nabla^M)$ and $(h, \nabla^N)$ are Riemannian. 

\vspace{0.5\baselineskip}

\begin{proof}
Using the condition $K^N\bigl(\tau(u),\tau(u)\bigr)=0$ 
and the conjugate symmetricity of $(h, \nabla^N)$, 
we have 
\begin{eqnarray*}
&&  
\Delta_g \, 
\| \tau(u) \|_h^2 \\
&& \quad
=h\left(\Delta^u\tau(u),\tau(u)\right)
+h\left(\tau(u), \bar{\Delta}^u\tau(u) \right)\\
&& \qquad
+\sum_{i=1}^m \left\{ 4h\left(\left(R^N-R^h\right)
\bigl(u_*e_i,\tau(u)\bigr)\tau(u),u_*e_i\right) \right. \\
&& \qquad \qquad
\left. 
+h(\nabla^u_{e_i}\tau(u), \nabla^u_{e_i}\tau(u)) 
+h(\overline{\nabla}^{\substack{\scalebox{0.3}{\phantom{i}}\\u}}_{e_i}\tau(u), \overline{\nabla}^{\substack{\scalebox{0.3}{\phantom{i}}\\u}}_{e_i}\tau(u))
\right\}. 
\end{eqnarray*}
By Green's theorem, Proposition $\ref{lapfor}$, 
and by the condition  $\operatorname{div}^g(\mathrm{tr}_gK^M)=0$, 
the conjugate symmetricity of $(h, \nabla^N)$ and $\tau_2(u)=0$, 
we have
\begin{eqnarray*}
0&=&\int_{M}\Delta_g \, 
\| \tau(u) \|_h^2  %h\left(\tau(u),\tau(u)\right) 
\,  d\mu_g \\
&=&2 \int_{M}h\left(\Delta^u\tau(u),\tau(u)\right)d\mu_g \\
&&\quad+4\sum_{i=1}^m\int_{M}h\left(
\left(R^N-R^h\right)\bigl(u_*e_i,\tau(u)\bigr)\tau(u),u_*e_i\right)d\mu_g\\
&&\qquad 
+\int_{M}
\{ 
\|\nabla^u\tau(u)\|^2_{g,h}+\|\overline{\nabla}^{\substack{\scalebox{0.3}{\phantom{i}}\\u}}\tau(u)\|^2_{g,h} \} d\mu_g \\
&=&-2\sum_{i=1}^m\int_{M}h\left(
U^N(u_*e_i,\tau(u))\tau(u),u_*e_i\right)d\mu_g\\
&&\quad+\int_{M}\{ \|\nabla^u\tau(u)\|^2_{g,h}+\|\overline{\nabla}^{\substack{\scalebox{0.3}{\phantom{i}}\\u}}\tau(u)\|^2_{g,h} \} d\mu_g. 
\end{eqnarray*}
Since the sectional curvature of $U^N$ is non-positive we have
    \begin{equation*}
        \nabla^u\tau(u)=\overline{\nabla}^{\substack{\scalebox{0.3}{\phantom{i}}\\u}}\tau(u)=0.
    \end{equation*}
Therefore, Lemma \ref{TauPar} completes the proof.  
\end{proof}

\vspace{0.5\baselineskip}

The following proposition states a local property 
in contrast to the previous ones. 
In the Riemannian setting it is obtained by \cite{MR2004799}. 

\vspace{0.5\baselineskip}

\begin{proposition}
Let $u:(M,g,\nabla^M)\to(N,h,\nabla^N)$ be a statistical hypersurface, 
that is, $(g,\nabla^M)$ is the statistical structure on $M$ 
induced by $u$ from $(h,\nabla^N)$, and $m=n-1$. 
Suppose that $(N,h,\nabla^N)$ is a trace-free and satisfies
\begin{equation}    \label{ricciinq}
    \operatorname{Ric}^N-\overline{\operatorname{Ric}}^{\substack{\scalebox{0.2}{\phantom{i}}\\N}}-2\operatorname{Ric}^h\geq0.
\end{equation}
If $u$ is a biharmonic map such that $K^N\bigl(\tau(u),\tau(u)\bigr)=0$, 
and $\|\tau(u)\|_h$ is constant on $M$, then $\tau(u)=0$.
\end{proposition}

\begin{proof}
Assume that $\tau(u)\neq0$. 
    
We take an orthonormal basis 
$\{u_* e_1, \ldots, u_* e_m, \|\tau(u)\|_h^{-1}\tau(u) \}$ for $h$, 
remarking that $u$ is an isometric immersion. 
Since $(h,\nabla^N)$ is trace-free,
\begin{eqnarray*}
\mathrm{tr}_gK^M
&=&
\sum_{i=1}^m h(u_* K^M(e_j, e_j), u_* e_i) e_i \\
&=&
-\sum_{i=1}^m\|\tau(u)\|^{-2}_h
h\left(K^N\bigl(\tau(u),\tau(u)\bigr),u_*e_i\right)e_i=0. 
\end{eqnarray*}
By using $\mathrm{tr}_gK^M=0$ and $K^N\bigl(\tau(u),\tau(u)\bigr)=0$ again, 
we calculate the statistical Laplacian of $\|\tau(u)\|^2_h$ as  
\begin{eqnarray}\label{laptau2}
            \frac{1}{2}\Delta \|\tau(u)\|^2_h
            &=&
            h\left(\bar{\Delta}^u\tau(u),\tau(u)\right)
            +\left\langle\nabla^u\tau(u),
            \overline{\nabla}^{\substack{\scalebox{0.3}{\phantom{i}}\\u}}\tau(u)\right\rangle_{g,h} \nonumber\\
            &=&
            h\left(\bar{\Delta}^u\tau(u),\tau(u)\right)
            +\|\widehat{\nabla}^u\tau(u)\|^2_{g,h}
            -\sum_{i=1}^m\|K^N\bigl(u_*e_i,\tau(u)\bigr)\|^2_h \nonumber \\
            &=&
            \dfrac{1}{2} \left\{
\operatorname{Ric}^N\bigl(\tau(u),\tau(u)\bigr)
-\overline{\operatorname{Ric}}^{\substack{\scalebox{0.2}{\phantom{i}}\\N}}\bigl(\tau(u),\tau(u)\bigr)
-2\operatorname{Ric}^h\bigl(\tau(u),\tau(u)\bigr)
\right\} \nonumber \\
&& \qquad %\qquad \qquad \qquad 
+\|\widehat{\nabla}^u\tau(u)\|^2_{g,h}, 
\end{eqnarray}
where $\widehat{\nabla}^u$ is the connection on $u^{-1}TN$ 
induced by $u$ from $\widehat{\nabla}=\nabla^h$.  
Indeed, the last equality is obtained as follows. 
For the third term, by \eqref{scurvature},  we have
    \begin{equation*}
        \begin{split}
    \sum_{i=1}^m\|K^N\bigl(u_*e_i,\tau(u)\bigr)\|^2_h
    &=\sum_{i=1}^mh\left(K^N_{\tau(u)}K^N_{u_*e_i}\tau(u),u_*e_i\right)\\
    &=\sum_{i=1}^mh\left(
    \lbrack K^N_{\tau(u)},K^N_{u_*e_i}\rbrack
    \tau(u),u_*e_i\right)\\
    &=-\frac{1}{2}\operatorname{Ric}^N\left(\tau(u),\tau(u)\right)
    -\frac{1}{2}\overline{\operatorname{Ric}}^{\substack{\scalebox{0.2}{\phantom{i}}\\N}}\left(\tau(u),\tau(u)\right)\\
    &\quad+\operatorname{Ric}^h\left(\tau(u),\tau(u)\right),
        \end{split}
    \end{equation*}
and for the first term, since $\tau_2(u)=0$, 
    \begin{equation*}
    \begin{split}
        h\left(\bar{\Delta}^u\tau(u),\tau(u)\right)
        &=\sum_{i=1}^mh\left(L^N(u_*e_i,\tau(u))u_*e_i,\tau(u)\right)\\
        &=-\sum_{i=1}^mh\left(R^N(\tau(u),u_*e_i)u_*e_i,\tau(u)\right)\\
        &=-\sum_{i=1}^mh\left(\overline{R}^{\substack{\scalebox{0.3}{\phantom{i}}\\N}}(u_*e_i,\tau(u))\tau(u),u_*e_i\right)\\
        &=-\overline{\operatorname{Ric}}^{\substack{\scalebox{0.2}{\phantom{i}}\\N}}\bigl(\tau(u),\tau(u)\bigr), 
    \end{split}
    \end{equation*}
from which we have \eqref{laptau2}. 
    
By \eqref{laptau2} and $(\ref{ricciinq})$ 
we obtain $\widehat{\nabla}^u\tau(u)=0$ since  $\|\tau(u)\|$ is constant. 
    
We will show that $\|\tau(u)\|_h=0$, which contradicts that $h$ is a positive definite metric, which completes the proof.  
\begin{equation*}
    \begin{split}
        0&=\sum_{i=1}^m h\left(u_*e_i,\widehat{\nabla}^u_{e_i}\tau(u)\right)\\
        &=-\sum_{i=1}^m h\left(\widehat{\nabla}^u_{e_i}u_*e_i,\tau(u)\right)\\
        &=-h\bigl(\tau(u),\tau(u)\bigr)+\sum_{i=1}^m h\left(K^N\bigl(u_*e_i,u_*e_i\bigr),\tau(u)\right)\\
        &=-\|\tau(u)\|_h^2-\|\tau(u)\|^{-2}_hh(K^N\bigl(\tau(u),\tau(u)\bigr),\tau(u))\\
        &=-\|\tau(u)\|_h^2.
    \end{split}
\end{equation*}
\end{proof}

\backmatter

\section*{Statements and Declarations}

\subsection*{Acknowledgments}
\bmhead{Funding}
This work was supported by the Japan Society for the Promotion of Science KAKENHI, Grant Number JP22K03279, and by the Japan Science and Technology agency, the establishment of university fellowships towards the creation of science technology innovation, Grant Number JPMJFS2101.

%\bmhead{Data availability}
%No data was gathered for this article.

%\bmhead{Declarations}

%\bmhead{Conflict of interest}
%The authors declare that they have no conflict of interest.

%\section*{Declarations}
%%\item Availability of data and materials
%\item Code availability 
%\item Authors' contributions
%\end{itemize}

%\noindent
%If any of the sections are not relevant to your manuscript, please include the heading and write `Not applicable' for that section. 

%%===================================================%%
%% For presentation purpose, we have included        %%
%% \bigskip command. please ignore this.             %%
%%===================================================%%
%\bigskip
%\begin{flushleft}%
%Editorial Policies for:
%
%\bigskip\noindent
%Springer journals and proceedings: \url{https://www.springer.com/gp/editorial-policies}
%
%\bigskip\noindent
%Nature Portfolio journals: \url{https://www.nature.com/nature-research/editorial-policies}
%
%\bigskip\noindent
%\textit{Scientific Reports}: \url{https://www.nature.com/srep/journal-policies/editorial-policies}
%
%\bigskip\noindent
%BMC journals: \url{https://www.biomedcentral.com/getpublished/editorial-policies}
%\end{flushleft}

%%===========================================================================================%%
%% If you are submitting to one of the Nature Portfolio journals, using the eJP submission   %%
%% system, please include the references within the manuscript file itself. You may do this  %%
%% by copying the reference list from your .bbl file, paste it into the main manuscript .tex %%
%% file, and delete the associated \verb+\bibliography+ commands.                            %%
%%===========================================================================================%%


\begin{thebibliography}{99}%

\bibitem{MR1823925}
Baues, O., Cort\'{e}s, V.: Realisation of special {K}\"{a}hler manifolds as parabolic spheres. Proc. Amer. Math. Soc. \textbf{129}(8), 2403-2407 (2001)

\bibitem{MR1143504}
Cheng, B.-Y.: Some open problems and conjectures on submanifolds of finite type. Soochow J. Math. \textbf{17}(2), 169-188 (1991)

\bibitem{FuruhataHasegawaSatoh2022}
Furuhata, H., Hasegawa I., Satoh N.: Chen invariants and statistical submanifolds. Commun. Korean Math. Soc. \textbf{37}(3), 851-864 (2022)

\bibitem{GuoyingJiang}
Jiang, G.Y.: 2-Harmonic maps and their first and second variational formulas. Translated from the Chinese by Hajime Urakawa. Note Mat. \textbf{28}, 209-232 (2009)

\bibitem{JiangTavakoliZhao}
Jiang, R., Tavakoli, J., Zhao, Y.: Laplacian operator on statistical manifold. Inf. Geom. \textbf{6}(1), 69-79 (2023)

\bibitem{Jostcimcir}
Jost, J., \c{S}im\c{s}ir, F.M.: Affine harmonic maps. Analysis (Munich). \textbf{29}(2), 185-197 (2009)

\bibitem{KobayashiNomizu1}
Kobayashi, S., Nomizu, K.: Foundations of differential geometry. {V}ol. {I} John Wiley \& Sons, Inc., New York (1963)

\bibitem{MR1317826}
Liu, H.L., Wang, C.P.: The centroaffine {T}chebychev operator. Results Math. \textbf{27}(1-2), 77-92 (1995)

\bibitem{NomizuSasaki}
Nomizu, K., Sasaki, T.: Affine differential geometry. Cambridge University Press, Cambridge (1994)

\bibitem{MR2004799}
Oniciuc, C.: Biharmonic maps between {R}iemannian manifolds. An. \c{S}tiin\c{t}. Univ. Al. I. Cuza Ia\c{s}i. Mat. (N.S.) \textbf{48}(2), 237-248 (2002)

\bibitem{MR3422914}
Opozda, B.: Bochner's technique for statistical structures. Ann. Global Anal. Geom. \textbf{48}(4), 357-395 (2015)

\bibitem{MR4010286}
{\c S}im\c{s}ir, F.M.: A note on harmonic maps of statistical manifolds. Commun. Fac. Sci. Univ. Ank. Ser. A1. Math. Stat. \textbf{68}(2), 1370-1376 (2019)

\bibitem{MR3329737}
Uohashi, K.: Harmonic maps relative to {$\alpha$}-connections. In: Nielsen, F. (ed.) Geometric theory of information, pp. 81-96. Springer Cham.(2014)

\bibitem{MR3889042}
Urakawa, H.: Geometry of biharmonic mappings. World Scientific Publishing Co. Pte. Ltd., Hackensack, N.J. (2019)

%Interscience Publishers, a division of John Wiley \& Sons, New York-London  1963 {\rm xi}+329 pp.
%
\end{thebibliography}
\end{document}